%% file: cross.tex
\documentclass[12pt]{amsart}
\usepackage{amscd,amsmath,amsthm,amssymb}
\usepackage{mathtools}
\usepackage{mathrsfs}
\usepackage{comment}
\usepackage[all,tips]{xy}
\usepackage{color}
\usepackage{circuitikz}
\usepackage{caption}
\usepackage{enumitem} 

\definecolor{cadmiumgreen}{rgb}{0.0, 0.42, 0.24}
\usepackage[
colorlinks, citecolor=cadmiumgreen,
pagebackref,
pdfauthor={Robin de Jong, Farbod Shokrieh}, 
pdftitle={Metric graphs, cross ratios, and Rayleigh's laws},
pdfstartview ={FitV},
]{hyperref}

\usepackage[
alphabetic,
backrefs,
msc-links,
nobysame,
lite,
]{amsrefs} 

\usepackage{tikz, float} 
\usetikzlibrary {positioning}

\usetikzlibrary{calc,decorations.markings}
\usetikzlibrary{shapes,snakes}

\usepackage[left=3.5cm,top=3.5cm,right=3.5cm]{geometry}
\newcommand*\isom{\xrightarrow{\sim}}

\def\RR{{\mathbb R}}

\def\EE{{\mathbb E}}

\def\bf{{\mathbf b}}

\def\Oc{{\mathbf O}}

\def\Sb{{\mathbf S}}

\def\xb{{\mathbf x}}

\def\Oc{{\mathcal O}}

\def\MA{{\mathcal A}}

\def\A{{\mathcal A}}

\def\Mc{{\mathcal M}}

\def\C{{\mathcal C}}

\def\B{{\mathbf B}}
\def\Bc{{\mathcal B}}
\def\Q{{\mathbf Q}}
\def\L{{\mathbf L}}
\def\D{{\mathbf D}}
\def\R{{\mathbf R}}
\def\I{{\mathbf I}}
\def\P{{\mathbf P}}
\def\M{{\mathbf M}}
\def\N{{\mathbf N}}
\def\Xib{{\mathbf \Xi}}
\def\F{\mathsf F}
\def\cc{\mathfrak c}
\def\ii{\mathfrak i}
\def\vv{\mathfrak v}

\def\opn#1#2{\def#1{\operatorname{#2}}}
\opn\depth{depth} 
\opn\codim{codim}
\opn\ini{in} 
\opn\LM{LM}
\opn\LC{LC}
\opn\NF{NF}
\opn\Merge{Merge}
\opn\sgn{sgn}
\opn\div{div} 
\opn\Div{Div} 
\opn\Pic{Pic}
\opn\Prin{Prin}
\opn\Del{Del}
\opn\op{op}
\opn\ends{ends}
\opn\indeg{indeg} 
\opn\sign{sign}
\opn\outdeg{outdeg}
\opn\red{red}
\opn\Spec{Spec} 
\opn\Supp{Supp} 
\opn\supp{supp} 
\opn\Ker{Ker} 
\opn\Coker{Coker} 
\opn\Hom{Hom}
\opn\Tor{Tor} 
\opn\id{id}
\opn\span{span}
\opn\Image{Image}
\opn\con{conv} 
\opn\relint{rel.int} 
\opn\vol{vol}
\opn\val{val}
\opn\Zh{Zh}
\opn\Ber{Ber}
\opn\Vor{Vor}
\opn\Vol{Vol}
\opn\Covol{Covol}
\opn\Jac{Jac}
\opn\Dir{Dir}
\opn\can{can}
\opn\syz{{\rm syz}}
\opn\spoly{{\rm spoly}}
\opn\LM{{\rm LM}}
\opn\lm{{\rm lm}}
\opn\lcm{{\rm lcm}} 
\opn\A{\mathcal A}
\opn\dist{dist}
\opn\pd{pd}
\opn\en{en}
\opn\PL{PL}
\opn\dmeasz{DMeas_0}
\opn\dmeas{DMeas}
\opn\T{T}
\opn\circu{circ}
\opn\cocirc{cocirc}
\opn\Proj{Proj}

\def\Implies{\ifmmode\Longrightarrow \else
        \unskip${}\Longrightarrow{}$\ignorespaces\fi}
\def\implies{\ifmmode\Rightarrow \else
        \unskip${}\Rightarrow{}$\ignorespaces\fi}
\def\iff{\ifmmode\Longleftrightarrow \else
        \unskip${}\Longleftrightarrow{}$\ignorespaces\fi}

\newtheorem{Theorem}{Theorem}[section]
\newtheorem{Lemma}[Theorem]{Lemma}
\newtheorem{Corollary}[Theorem]{Corollary}
\newtheorem{Proposition}[Theorem]{Proposition}

\theoremstyle{remark}
\newtheorem{Remark}[Theorem]{Remark}

\theoremstyle{definition}
\newtheorem{Example}[Theorem]{Example}
\newtheorem{Definition}[Theorem]{Definition}

\def\qed{\ifhmode\textqed\fi
      \ifmmode\ifinner\quad\qedsymbol\else\dispqed\fi\fi}
\def\textqed{\unskip\nobreak\penalty50
       \hskip2em\hbox{}\nobreak\hfil\qedsymbol
       \parfillskip=0pt \finalhyphendemerits=0}
\def\dispqed{\rlap{\qquad\qedsymbol}}

\numberwithin{equation}{section}

\tikzstyle{Cwhite}=[scale = .8,circle, fill = white, minimum size=3mm] 
\tikzstyle{Cgray}=[scale = .4,circle, fill = gray, minimum size=3mm] 
\tikzstyle{Cblack2}=[scale = .4,circle, fill = black, minimum size=5mm] 
\tikzstyle{Cblack}=[scale = .7,circle, fill = black, minimum size=3mm]
\tikzstyle{C0}=[scale = .9,circle, fill = black!0, inner sep = 0pt, minimum size=3mm]
\tikzstyle{C1}=[scale = .7,circle, fill = black!0, inner sep = 0pt, minimum size=3mm]
\tikzstyle{Cred}=[scale = .4,circle, fill = red, minimum size=3mm]

\ctikzset{bipoles/resistor/height=0.15}
\ctikzset{bipoles/resistor/width=0.4}

\begin{document}

\title{Metric graphs, cross ratios, and Rayleigh's laws}

\author{Robin de Jong}
\email{\href{mailto:rdejong@math.leidenuniv.nl}{rdejong@math.leidenuniv.nl}}

\author {Farbod Shokrieh}
\email{\href{mailto:farbod@math.cornell.edu}{farbod@math.cornell.edu}, \href{mailto:farbod@math.ku.dk}{farbod@math.ku.dk}}

\subjclass[2010]{
\href{https://mathscinet.ams.org/msc/msc2010.html?t=94C05}{94C05},
\href{https://mathscinet.ams.org/msc/msc2010.html?t=14T05}{14T05},
\href{https://mathscinet.ams.org/msc/msc2010.html?t=35J05}{35J05},
\href{https://mathscinet.ams.org/msc/msc2010.html?t=05C50}{05C50}}


\date{\today}

\begin{abstract}
We study a notion of cross ratios on metric graphs and electrical networks. We show that several known results immediately follow from the basic properties of cross ratios. We show that the projection matrices of Kirchhoff have nice (and efficiently computable) expressions in terms of cross ratios. Finally we prove a very general version of Rayleigh's law, relating energy pairings and cross ratios before and after contracting an edge segment. As a corollary, we obtain a quantitative version of {\em Rayleigh's monotonicity law} for effective resistances. Another consequence is an explicit description of the behavior of the potential kernel of the Laplacian operator under contractions.
\end{abstract}

\maketitle

\setcounter{tocdepth}{1}
\tableofcontents

\section{Introduction} \label{sec:Intro}
\renewcommand*{\theTheorem}{\Alph{Theorem}}

\subsection{Background}
If the resistances in an electrical network are decreased, the effective resistance between any two points can only decrease. This is known as {\em Rayleigh's monotonicity law}, first stated by Lord Rayleigh in \cite{Rayleigh}. Maxwell, in his {\em Treatise on Electricity and Magnetism}, regards this law as `self-evident' (\cite[Chapter VIII, Paragraph 306]{Maxwell}). The formal proof is rather straightforward as well (see, e.g., \cite[Chapter 4]{ds}); one first realizes that an effective resistance may be thought of as a `distance' (see Remark~\ref{rmk:Thomson}). Then Rayleigh's law boils down to the following elementary fact in linear algebra: if $W_1 \subseteq W_2 \subseteq V$ are normed vector spaces and $\mathbf{p} \in V$ then $\dist(\mathbf{p}, W_1) \geq \dist(\mathbf{p}, W_2)$. One of our results is a {\em quantitative} (or {\em effective}) version of this law (see Corollary~\ref{cor:EffRayleighIntro} below). We will describe precisely, in terms of the notion of `cross ratios',  the amount by which effective resistances decrease. 

\subsection{Our contributions}

Let $\Gamma$ be a metric graph (i.e. a length metric space homeomorphic to a topological graph) which is compact and connected. We may think of $\Gamma$ as a (resistive) electrical network. Consider the {\em $j$-function} on $\Gamma$ defined informally as follows: for $x, y, z \in \Gamma$ let $j_z(x,y; \Gamma)$ denote the electric potential at $x$ when one unit of current enters the network $\Gamma$ at $y$ and exits at $z$, with $z$ `grounded' (i.e. has zero potential). The effective resistance between two points $x, y \in \Gamma$ is defined as $r(x,y; \Gamma) \coloneqq j_y(x,x; \Gamma)$.

The $j$-function is the kernel of integration that inverts the (distributional) Laplacian operator  $\Delta$, so it appears naturally in the harmonic analysis of $\Gamma$. There is a very different interpretation of $j$-functions in the language of Gromov's theory of hyperbolic spaces: let $(x|y)_z$ denote the {\em Gromov product}  on $\Gamma$ with respect to the effective resistance distance function $r \colon \Gamma \times \Gamma \rightarrow \RR$ (see \S\ref{sec:GromovProd}). Then $(x|y)_z = j_z(x,y ; \Gamma)$ (see Lemma~\ref{lem:GromovProd} and Example~\ref{ex:GromovProd}). Motivated by this latter point of view, we define the {\em cross ratio function} $\xi$ on $\Gamma$ as follows: fix a point $q \in \Gamma$. For $x,y,z,w \in \Gamma$ let
\[
\xi(x,y , z,w ; \Gamma) \coloneqq j_q(x,z; \Gamma) +j_q(y,w; \Gamma) - j_q(x,w ; \Gamma) - j_q(y,z ; \Gamma) \,.
\]

The resulting function is easily seen to be independent of the base point $q$ (see Lemma~\ref{lem:CrossProp}). The $j$-function and the $r$-function are evaluations of the cross ratio function: 
\[ j_z(x,y ; \Gamma) = \xi(x,z , y,z ; \Gamma) \quad , \quad r(x,y ; \Gamma) = \xi(x,y , x,y ; \Gamma)\, .\]
 We further argue that many important properties and formulas in the theory of electrical networks are immediate consequences of the basic properties of cross ratios. For example, the `reciprocity theorem' in electrical networks is a consequence of the fact that cross ratios are independent of the choice of base points (see Example~\ref{ex:recip}).

The cross ratio function is, in turn, an evaluation of the {\em energy pairing}. Let $\dmeasz(\Gamma)$ denote the real vector space of discrete measures  on $\Gamma$ with zero total mass. For $\nu_1 , \nu_2 \in \dmeasz(\Gamma)$, the energy pairing is defined by
\[
\langle \nu_1 , \nu_2 \rangle_{\en}^{\Gamma} \coloneqq \int_{\Gamma \times \Gamma} j_q (x,y;\Gamma) \mathrm{d}\nu_1(x) \mathrm{d}\nu_2(y) \, ,
\]
which is, again, independent of the base point $q \in \Gamma$. One can easily check 
\[ \xi(x,y , z,w ; \Gamma) = \langle \delta_x-\delta_y, \delta_z-\delta_w \rangle_{\en}^{\Gamma}\, .\] 
A vast generalization of the reciprocity theorem is the statement that the energy pairing can be computed using {\em any} generalized inverse of the Laplacian operator. See Lemma~\ref{lem:pairings}, Example~\ref{ex:recip}, and Remark~\ref{rmk:gen_recip}.

We can now state our {\em Rayleigh's law for energy pairings}. For an edge segment $e$ of $\Gamma$, let $\Gamma/e$ denote the network obtained by `short-circuiting' the segment $e$.
\begin{Theorem}\label{thm:GenRayleighEnergyIntro} (=Theorem~\ref{thm:GenRayleighEnergy})
Let $\Gamma$ be a metric graph. Let $e$ be an edge segment of $\Gamma$ with endpoints $\partial e = \{e^-, e^+\}$, and let $\nu_1, \nu_2 \in \dmeasz(\Gamma)$.  Then
\[
\langle \nu_1, \nu_2 \rangle_{\en}^{\Gamma/e} = 
\langle \nu_1, \nu_2 \rangle_{\en}^{\Gamma} 
- \frac{\langle \nu_1, \delta_{e^+} - \delta_{e^-} \rangle_{\en}^{\Gamma} \ \,
\langle \delta_{e^+} - \delta_{e^-}, \nu_2 \rangle_{\en}^{\Gamma}}
{r(e^-, e^+; \Gamma)}\, .
\]
\end{Theorem}

One immediately also obtains {\em Rayleigh's laws for cross ratios, $j$-functions, and $r$-functions}, as these are all evaluations of energy pairings. For example, here is a quantitative version of Rayleigh's monotonicity law promised earlier.
\begin{Corollary}\label{cor:EffRayleighIntro}(=Corollary~\ref{cor:EffRayleigh})
\[
r(x,y ; \Gamma/e) =r(x,y ; \Gamma) -  \frac{\xi(x,y , e^-, e^+; \Gamma) ^2}{r(e^-, e^+; \Gamma)} \, .
\]
In particular, $ r(x,y ; \Gamma/e) \leq r(x,y ; \Gamma)$.
\end{Corollary}

We now describe the main ingredients of the proof of Theorem~\ref{thm:GenRayleighEnergyIntro}. Let $G$ be a model of $\Gamma$, which is a finite graph together with a length function $\ell$ on its edge set. Fix an orientation $\Oc$ on $G$. The space of $1$-chains $C_1(G,\RR)\simeq \bigoplus_{e \in \Oc}  \RR e$ is endowed with a canonical bilinear form: for distinct $e,f \in \Oc$ we let $[e,e] = \ell(e)$ and $[e,f] = 0$. The first homology group $H_1(G, \RR)$ is naturally a subspace of the inner product space $C_1(G,\RR)$. The theory of electrical networks is essentially the study of the orthogonal projections (see \S\ref{sec:networkproblems}):
\[
\begin{aligned}
\pi_G \colon & C_1(G,\RR) \twoheadrightarrow H_1(G, \RR) \, , \\
\pi'_G \colon &C_1(G,\RR) \twoheadrightarrow H_1(G, \RR)^{\perp} \, .
\end{aligned}
\]

Kirchhoff, in his seminal paper \cite{Kirchhoff}, gave a beautiful description of these projections (in the basis given by $\Oc$) as a certain average over spanning trees. We will review his description in \S\ref{sec:projST}. It turns out there is another convenient description in terms of cross ratios. Let $m$ denote the number of edges of $G$, and let $\Xib$ be the $m \times m$ {\em matrix of cross ratios}:
\[
\Xib \coloneqq \left(\xi(e^-, e^+ , f^- , f^+) \right)_{e,f \in \Oc} \, . 
\]
Let $\D$ be the $m \times m$ diagonal matrix whose diagonal entries are $\ell(e)$ for $e \in \Oc$, and let $\I$ be the identity matrix.
\begin{Theorem}\label{thm:crossprojIntro}(=Proposition~\ref{prop:crossproj})
In the basis given by $\Oc$,  the matrix of $\pi_G$ is $\I- \D^{-1} \Xib$, and the matrix of $\pi'_G$ is $\D^{-1} \Xib$.
\end{Theorem}
Theorem~\ref{thm:crossprojIntro} is proved by a very straightforward linear algebraic argument. However, these expressions in terms of cross ratios are quite useful for applications. Moreover, unlike Kirchhoff's classical description, our description is very efficient for computations: both projection matrices can be computed in time at most $O(n^\omega)$, where $n$ is the number of vertices of $G$ and $\omega$ is the exponent for the matrix multiplication algorithm (see Remark~\ref{rmk:efficient}).

A key observation is that the energy pairing can be computed using the projection matrix $\pi'_G$. See Proposition~\ref{prop:energyformula} for the  precise statement. 
It turns out that the change in Kirchhoff's projection matrices before and after contraction of an edge segment can be explicitly described in terms of cross ratios: the matrix of $\pi'_{G/e}$, with respect to the basis given by $\Oc$, is
\[
\Sb = \D^{-1}\Xib - \frac{1}{r(e^-, e^+)} \D^{-1} (\Xib [e]) (\Xib [e])^{\T}   \, ,
\]
where $[e]$ denotes the characteristic vector of $e \in \Oc$. See Proposition~\ref{prop:diagram} and the proof of Theorem~\ref{thm:GenRayleighEnergy}. This, together with Proposition~\ref{prop:energyformula}, leads us to the generalized versions of Rayleigh's law.

\medskip

Our work was partially motivated by questions in non-archimedean analytic (and tropical) geometry. Any metric graph arises as a skeleton of some Berkovich analytic curve over a non-archimedean field. The potential theory on such curves is more or less the same as the study of electrical networks. In  \cite{dJSh} we make extensive use of various descriptions of projection matrices given in the present paper, as well as our generalized Rayleigh's laws, to compute invariants arising from Arakelov geometry.

\subsection{Structure of the paper}
In \S\ref{sec:graphs} we set the notation and terminology for graphs and electrical networks and discuss their correspondences. 
In \S\ref{sec:laplacian} various notions of Laplacian operators and their compatibilities are reviewed.  
The notion of $j$-functions and their relation to Gromov products are discussed in \S\ref{sec:potential}.   
In \S\ref{sec:energy} we study the notions of energy and Dirichlet pairings, and discuss their relationship. We show that energy pairings can be computed using {\em any} generalized inverse of the Laplacian matrix. In \S\ref{sec:cross} cross ratios are defined and some of their basic properties are established. 
In \S\ref{sec:projs} we will review Kirchhoff's classical work on projection matrices arising in electrical networks. We then present our description in terms of cross ratios (Theorem~\ref{thm:crossprojIntro}). 
In \S\ref{sec:Rayleigh} we state and prove our generalized and quantitative versions of Rayleigh's laws (including Theorem~\ref{thm:GenRayleighEnergyIntro} and Corollary~\ref{cor:EffRayleighIntro}), relating various evaluations of energy pairings before and after edge contractions.

\medskip

Throughout, we have made an attempt to keep the paper as self-contained as possible. For example we review classical topics such as Kirchhoff's description of projection matrices. But, in doing so, we have tried to present a new perspective on how one may think about various statements. For a good highlight of our attempt, the reader might be interested in our description of Ohm's law in \S\ref{sec:Ohm}.


\renewcommand*{\theTheorem}{\arabic{section}.\arabic{Theorem}}
\section{Graphs and networks} \label{sec:graphs}
\subsection{Weighted graphs}
By a {\em weighted graph} we mean a finite weighted connected multigraph $G$ with no loop edges. We denote the set of vertices of $G$ by $V(G)$ and the set of edges of $G$ by $E(G)$, and let $n = |V (G)|$ and $m = |E(G)|$. We assume both $V(G)$ and $E(G)$ to be non-empty, and this gives $n \geq 2$. The weights of edges are determined by a {\em length function} 
\[ \ell \colon E(G) \to \RR_{>0}\, .\] 
We let $\mathbb{E}(G) = \{e, \bar{e} \colon e \in E(G)\}$ denote the set of {\em oriented edges}. We have $\bar{\bar{e}} = e$. An {\em orientation} $\Oc$ on $G$ is a partition $\mathbb{E}(G) = \Oc \cup \overline{\Oc}$, where $\overline{\Oc} = \{\bar{e} \colon e \in \Oc\} $. We have an obvious extension of the length function 
\[\ell \colon \mathbb{E}(G) \rightarrow  \RR_{>0}\] 
by requiring $\ell(e) = \ell(\bar{e})$. There is a map
$\mathbb{E}(G) \rightarrow V(G) \times V(G)$ sending an oriented edge $e$ to $(e^+, e^-)$. 

\subsection{Metric graphs and models}
A {\em metric graph} is a pair $(\Gamma, \ell)$ consisting of a compact connected topological graph $\Gamma$, together with an {\em inner metric} $\ell$. We will always assume $\Gamma$ is not a single point. In this case one can alternatively define a metric graph as a compact connected metric space $\Gamma$ such that every point has a neighborhood isometric to a star-shaped set, endowed with the path metric. We often assume $\ell$ is implicitly defined, and refer to $\Gamma$ as the metric graph.

The points of $\Gamma$ that have valency different from $2$ are called the {\em branch points} of $\Gamma$. A {\em vertex set} for $\Gamma$ is a finite set  $V \subset \Gamma$ containing all branch points such that each connected component $c$ of $\Gamma \setminus V$ has the property that the closure of $c$ in $\Gamma$ is isometric with a closed interval. Each vertex set $V$ of $\Gamma$ naturally determines a weighted graph $G$ with non-empty set of vertices $V(G)=V$ and with non-empty set of edges $E(G)$ given by the connected components of $\Gamma \setminus V$. We call such a weighted graph $G$ determined by a vertex set a {\em model} of $\Gamma$. 

Given a vertex set $V$ of $\Gamma$, the closure $e$ in $\Gamma$ of a connected component of $\Gamma \setminus V$ is called an {\em edge segment} of $\Gamma$. Note that there is a natural bijective correspondence between the set of edge segments determined by $V$ and the edge set $E(G)$ of the associated weighted graph. Given an edge segment $e$ of $\Gamma$ we denote by $\partial e = \{e^-,e^+\} \subset V$ the set of boundary points of $e$. We use the notation $\partial e = \{e^-,e^+\}$ for the set of boundary points even when there is no orientation present. We hope that this does not lead to confusion.

\subsection{Electrical networks}
Let $\Gamma$ be a metric graph and $G$ be a model of $\Gamma$. We will think of $\Gamma$ (or $G$) as an electrical network in which each edge $e \in E(G)$ is a resistor having resistance $\ell(e)$. The vertex set corresponding to $G$ may be thought of as the set of {\em access points} of the network, i.e. the points at which the external current or voltage sources can be attached or measurements can be done. See Figure~\ref{fig:graphs}.

When studying the `potential theory' on a metric graph $\Gamma$, it is convenient to always fix an (arbitrary) model $G$, and think of it as an electrical network as above. This will often allow us to give concrete linear algebraic formulas for quantities and functions of interest. We refer to \cite[Chapter II]{Bollobas} and \cite{Biggs} for an introduction to the theory of electrical networks from this point of view.

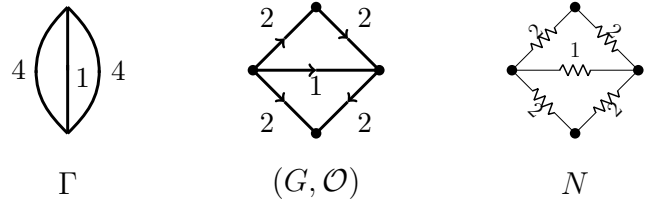
\begin{figure}[h!]
$$
\begin{xy}
(0,0)*+{
	\scalebox{.7}{$
	\begin{tikzpicture}
	\draw[black, ultra thick, -] (0,1.2) to [out=-45,in=90] (.6,-.05);
	\draw[black, ultra thick] (.6,0) to [out=-90,in=45] (0,-1.2);
	\draw[black, ultra thick, -] (0,1.2) to [out=-135,in=90] (-.6,-.05);
	\draw[black, ultra thick] (-.6,0) to [out=-90,in=135] (0,-1.2);
	\draw[black, ultra thick, -] (0,1.2) -- (0,0);
	\draw[black, ultra thick] (0,0.1) -- (0,-1.2);
	\end{tikzpicture}
	$}
};
(-6.5,0)*+{\mbox{{\smaller $4$}}};
(2,-1)*+{\mbox{{\smaller $1$}}};
(6.75,-0.05)*+{\mbox{{\smaller $4$}}};
(0,-15)*+{\Gamma};
\end{xy}
\ \ \ \ \ \ \ \ \ \ 
\begin{xy}
(0,0)*+{
	\scalebox{.7}{$
	\begin{tikzpicture}
	\draw[black, ultra thick, ->] (-1.2,0) to (-.6,.6);
	\draw[black, ultra thick] (-.6,.6) to (0,1.2);
	\draw[black, ultra thick, ->] (-1.2,0) to (0,0);
	\draw[black, ultra thick] (0,0) to (1.2,0);
	\draw[black, ultra thick, ->] (-1.2,0) to (-.6,-.6);
	\draw[black, ultra thick] (-.6,-.6) to (0,-1.2);
	\draw[black, ultra thick, ->] (0,1.2) to (.6,.6);
	\draw[black, ultra thick] (.6,.6) to (1.2,0);
	\draw[black, ultra thick, ->] (1.2,0) to (.6,-.6);
	\draw[black, ultra thick] (.6,-.6) to (0,-1.2);
	\fill[black] (0,1.2) circle (.1);
	\fill[black] (0,-1.2) circle (.1);
	\fill[black] (1.2,0) circle (.1);
	\fill[black] (-1.2,0) circle (.1);
	\end{tikzpicture}
	$}
};
(-6.5,7)*+{\mbox{{\smaller $2$}}};
(0,-2)*+{\mbox{{\smaller $1$}}};
(6.5,7)*+{\mbox{{\smaller $2$}}};
(-6.5,-7)*+{\mbox{{\smaller $2$}}};
(6.5,-7)*+{\mbox{{\smaller $2$}}};
(0,-15)*+{(G, \Oc)};
\end{xy}
\ \ \ \ \ \ \ \ \ \ 
\begin{xy}
(0,0)*+{
	\scalebox{.7}{$
	\begin{tikzpicture}
	\draw[black] (-1.2,0) to[R=$2$] (0,1.2);
	\draw[black] (-1.2,0) to[R=$1$] (1.2,0);
	\draw[black] (0,-1.2) to[R=$2$] (-1.2,0);
	\draw[black] (0,1.2) to[R=$2$] (1.2,0);
	\draw[black] (1.2,0) to[R=$2$] (0,-1.2);
	\fill[black] (0,1.2) circle (.1);
	\fill[black] (0,-1.2) circle (.1);
	\fill[black] (1.2,0) circle (.1);
	\fill[black] (-1.2,0) circle (.1);
	\end{tikzpicture}
	$}
};
(0,-15)*+{N};
\end{xy}
\ \ \ \ \ \ \ \ \ \ 
\!\!\!\!\!\!\!\!\!\!\!\!\!\!\!\!\!\!\!
$$
\caption{A metric graph $\Gamma$, a model $G$ with an orientation $\Oc$, and the corresponding electrical network $N$.}\label{fig:graphs}
\end{figure}

\section{Laplacian operators} \label{sec:laplacian}
\subsection{Distributional Laplacians on metric graphs}
Let $\Gamma$ be a metric graph. We let $\PL(\Gamma)$ be the real vector space consisting of all continuous piecewise affine real valued functions on $\Gamma$ that can change slope finitely many times on each closed edge segment. 
Let $\Delta$ be the Laplacian operator in the sense of distributions; for $\phi \in \PL(\Gamma)$, its Laplacian $\Delta(\phi)$ is the discrete measure
\[
\Delta(\phi) = \sum_{p \in \Gamma} \sigma_p(\phi) \delta_p \, ,
\]
where $\delta_p$ is the usual delta (Dirac) measure centered at $p$, and $\sigma_p(\phi)$ is the sum of incoming slopes of $\phi$ in all tangent directions at $p$. 

Let $\dmeasz(\Gamma)$ denote the real vector space of discrete measures $\nu$ on $\Gamma$ with $\nu(\Gamma) = 0$. Any $\nu \in \dmeasz(\Gamma)$ is of the form $\nu = \sum_{p \in \Gamma}a_p\delta_p$, where $a_p \in \RR$, all but finitely $a_p$'s are zero, and  $\sum_{p \in \Gamma}a_p = 0$. 
One can easily check $\Delta(\phi) \in \dmeasz(\Gamma)$. Let $\RR \subset \PL(\Gamma)$ denote the space of constant functions on $\Gamma$. Then $\Delta$ induces an isomorphism of vector spaces $ \PL(\Gamma) / \RR \xrightarrow{\sim} \dmeasz(\Gamma)$. 

\subsection{Combinatorial Laplacians on weighted graphs}\label{sec:comb_lap}
Let $G$ be a weighted graph. We denote by $\Mc(G) = \Hom(V(G),\RR) = C^0(G, \RR)$ the real vector space of of $\RR$-valued functions on $V(G)$. Let $\Delta$ be the (combinatorial) Laplacian operator; for $\psi \in \Mc(G)$, its Laplacian $\Delta(\psi)$ is the measure
\[
\Delta(\psi) = \sum_{p \in \Gamma} \Delta_p(\psi) \delta_p \, ,
\]
where
\[
\Delta_p(\psi) \coloneqq \sum_{e = \{p, v\} \in E(G)} (\psi(p) - \psi(v))/{\ell(e)} \, .
\]

Let $\dmeasz(G)$ denote the real vector space of discrete measures $\nu$ on $V(G)$ with $\nu(V(G)) = 0$. Any $\nu \in \dmeasz(G)$ is of the form $\nu = \sum_{p \in V(G)}a_p\delta_p$, with $a_p \in \RR$ and  $\sum_{p \in V(G)}a_p = 0$. 
One can easily check $\Delta(\psi) \in \dmeasz(G)$. Let $\RR \subset \Mc(G)$ denote the space of constant functions on $V(G)$. Then $\Delta$ induces an isomorphism of vector spaces $ \Mc(G) / \RR \xrightarrow{\sim} \dmeasz(G)$.

\subsection{Compatibilities and Laplacian matrices} \label{sec:compatible}
The (distributional) Laplacian $\Delta$ and the (combinatorial) Laplacian $\Delta$ are compatible in the following sense. Let $\phi \in \PL(\Gamma)$ and let $G$ be a model of $\Gamma$ such that $V(G)$ contains all those points of $\Gamma$ at which $\psi$ changes slopes. Let $\psi \in \Mc(G)$ denote the function obtained from $\phi$ by restriction. Then $\sigma_p(\phi) = \Delta_p(\psi)$.

The (combinatorial) Laplacian operator on a weighted graph $G$ can be conveniently presented by its {\em Laplacian matrix}. Let $\{v_1, \ldots , v_n\}$ be a labeling of $V(G)$. The Laplacian matrix $\Q$ associated to $G$ is the $n \times n$ matrix $\Q = (q_{ij})$, where for $i \ne j$
\[
q_{ij} = - \sum_{e =\{v_i , v_j\} \in E(G)}{{1}/{\ell(e)}} \, .
\]
The diagonal entries are determined by forcing the matrix to have zero-sum rows: 
\[
q_{ii} =  -\sum_{j\ne i} {q_{ij}} = \sum_{e = \{v_i, v\} \in E(G)}{{1}/{\ell(e)}} \, .
\]
It is well-known that $\Q$ is symmetric, has rank $n-1$ and that its kernel consists of constant functions (see, e.g., \cite{Biggs}).

Let $H \subset \RR^n$ be the subspace of vectors $(x_1,\ldots,x_n)$ such that $\sum_{i=1}^n x_i = 0$. The labeling $\{v_1, \ldots , v_n\}$ allows one to fix isomorphisms 
\[[\,\cdot\,] \colon \Mc(G) \xrightarrow{\sim} \RR^n \quad , \quad [\,\cdot\,] \colon \dmeasz(G) \xrightarrow{\sim} H \, .\]
Then for all $\psi \in \Mc(G)$ we have
\[
[\Delta(\psi)] = \Q [\psi]\, .
\]

The Laplacian matrix of $G$ can also be expressed in terms of the {\em incidence matrix} of $G$. Let $\{v_1, \ldots , v_n\}$ be a labeling of $V(G)$ as before. Fix an orientation $\Oc = \{e_1, \ldots, e_m\}$ on $G$. The incidence matrix $\B$ associated to $G$ is the $n \times m$ matrix $\B=(b_{ij})$, where 
\[
b_{ij} = 
\begin{cases}
+1 &\text{ : } e_j^{+} = v_i \\
-1 &\text{ : } e_j^{-} = v_i \\
0  &\text{ : otherwise. } 
\end{cases}
\]
Let $\D$ denote the $m \times m$ diagonal matrix with diagonal entries $\ell(e_i)$ for $e_i \in \Oc$. We have 
\begin{equation}\label{eq:factor}
\Q = \B \D^{-1}\B^{\T} \, ,
\end{equation}
where $(\cdot)^{\T}$ denotes the usual matrix transpose operation.

\section{Potential kernels} \label{sec:potential}
\subsection{The $j$-function} \label{sec:jfunc}
A fundamental solution of the Laplacian is given by $j$-functions. We follow the notation of \cite{cr}. See also \cite{zhang93, br, bffourier ,fm, sw} for more details and formulas. 

Let $\Gamma$ be a metric graph and fix two points $y,z \in \Gamma$. We denote by $j_z(\cdot \, , y ; \Gamma)$ the unique function in $\PL(\Gamma)$ satisfying:
\begin{itemize}
	\item[(i)] $\Delta \left(j_z(\cdot \, , y; \Gamma)\right )= \delta_y - \delta_z$, 
	\item[(ii)] $j_z(z,y; \Gamma) = 0$.
\end{itemize}
If the metric graph $\Gamma$ is clear from the context, we will write $j_z(x,y)$ instead of $j_z(x,y; \Gamma)$. Observe that $j_z(\cdot,y)$ is a harmonic function on $\Gamma \setminus \{ y,z \}$.

\begin{Remark} \label{rmk:jelectric}
If we think of $\Gamma$ as an electrical network, the function $j_z(x, y)$ has a nice interpretation (Figure~\ref{fig:jfunction}): it denotes the electric potential at $x$ if one unit of current enters the network at $y$ and exits at $z$, with $z$ `grounded' (i.e. has zero potential). See \S\ref{sec:Dirichlet}. 

\ctikzset{bipoles/length=.7cm}
\begin{figure}[h!]
\begin{tikzpicture}
\node [cloud, draw,cloud puffs=10.2,cloud puff arc=120, aspect=2, inner ysep=1em, gray] (cloud) at (0, 0) {};
	\fill[black] (0,-.5) circle (.1);
	\fill[black] (.7,.3) circle (.1);
	\fill[black] (-.8,0) circle (.1);
	\draw (0,-.5) to (0,-1) node[ground]{};
	\draw (0,-.9) to (1.8,-.9) to[dcisource] (1.8, .3) to (.7,.3);
	\draw (-1.8,-.9) [o-] to (0,-.9);
	\draw (-1.8,0) [o-] to (-.8,0);
	\fill[black] (0,-.9) circle (.05);
	\node at (-2, -.45) {\tiny{$j_z(x,y)$}};
	\node at (-2, -.9) {$-$};
	\node at (-2, 0) {$+$};
	\node at (0, -.2) {\tiny{$z$}};
	\node at (.7, 0) {\tiny{$y$}};
	\node at (-.8, .3) {\tiny{$x$}};
	\node at (2.2, -.3) {\tiny{$1$}};
\end{tikzpicture}
\caption{Electrical network interpretation of the $j$-function.}\label{fig:jfunction}
\end{figure}
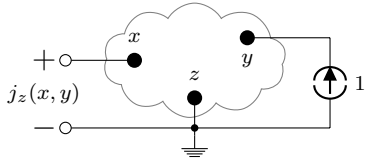
\end{Remark}
It follows immediately from electrical network theory that $j_z(\cdot \, , y ; \Gamma) \in \PL(\Gamma)$ exists and is unique. For a modern exposition, as well as an explicit integral formula for the $j$-function, see \cite[\S3.3]{sw} (see, also, Remark~\ref{rmk:metric_cross}).

\begin{Remark} \label{rmk:jprops}
The following properties of the $j$-function are expected from the electrical network interpretation, and are easy to prove (see, e.g., \cite[Lemma 2.10]{cr})
\begin{itemize}
\item[(i)] $j_z(x,y)$ is jointly continuous in all three variables $x,y,z \in \Gamma$.
\item[(ii)] $j_z(x,y) = j_z(y,x)$.
\item[(iii)] $0 \leq j_z(x,y) \leq j_z(x,x)$.
\item[(iv)] $j_z(x,x) = j_x(z,z)$.
\end{itemize}
\end{Remark}
\subsection{Effective resistance and Gromov products} \label{sec:GromovProd}
Following the electrical network interpretation, it makes sense to make the following definition. 

\begin{Definition}
The {\em effective resistance} between two points $x, y \in \Gamma$ is
\[
r(x,y) \coloneqq j_y(x,x) = j_x(y,y) \, .
\]
If we want to clarify the effective resistance is measured on $\Gamma$, we will use the notation $r(x,y; \Gamma)$ instead. 
\end{Definition}

Let $X$ be a set and let $d \colon X \times X \to \RR$ be a symmetric map.  For $x, y, z \in X$, one defines (see, e.g., \cite[Definition~1.19]{bh}) the {\em Gromov product} $(x|y)_z$ of $x$ and $y$ relative to $z$ by the formula
\[(x|y)_z \coloneqq\frac{1}{2}\left(d(x, z) + d(y, z) - d(x, y)\right)\, .\]
Note that $d$ satisfies the triangle inequality if and only if for all $x, y, z \in X$ the Gromov product $(x|y)_z$ is non-negative. 
\begin{Lemma}\label{lem:GromovProd}
Let $\Gamma$ be a metric graph, and let $r \colon \Gamma \times \Gamma \rightarrow \RR$ be the effective resistance function. Then $j_z(x,y)$ is precisely the Gromov product $(x|y)_z$ applied to the pair $(\Gamma, r)$. 
\end{Lemma}
\begin{proof}
For `tripods' the equality $j_z(x,y)=(x|y)_z$ is immediate (as observed in \cite[Remark B.7]{BR_Book}) -- see Figure~\ref{fig:gromov}. The general case follows from the tripod version by applying standard `circuit reduction' techniques (see, e.g., \cite[\S5.5]{circuit2011}). 
\end{proof}

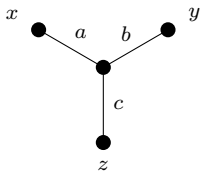
\begin{figure}[h!]
\begin{tikzpicture}
	\fill[black] (.86,.5) circle (.1);
	\fill[black] (0,-1) circle (.1);
	\fill[black] (-.86, .5) circle (.1);
	\fill[black] (0,0) circle (.1);
	\draw[black] (0,-1) to (0,0) to (.86,.5);
	\draw[black] (0,0) to (-.86,.5);
	\node at (1.21,.7 ) {\tiny{$y$}};
	\node at (-1.21,.7) {\tiny{$x$}};
	\node at (0, -1.3) {\tiny{$z$}};
	\node at (-.3, .45) {\tiny{$a$}};
	\node at (.3, .45) {\tiny{$b$}};
	\node at (.2, -.5) {\tiny{$c$}};
\end{tikzpicture}
\caption{A `tripod' with edge lengths $a,b,c$. Clearly $r(x,z) = a+c $, $r(y,z) = b+c $, $r(x,y) = a+b$, and $j_z(x,y) = c$.}\label{fig:gromov}
\end{figure}

\begin{Remark} \phantomsection \label{rmk:gromov} 
\begin{itemize}
\item[]
\item[(i)] We will give a direct proof (avoiding circuit reductions) of Lemma~\ref{lem:GromovProd} using cross ratios in Example~\ref{ex:GromovProd}. 
\item[(ii)]  By Remark~\ref{rmk:jprops}~(iii) we have $j_z(x,y) \geq 0$. Therefore Lemma~\ref{lem:GromovProd} has the immediate corollary that the effective resistance function satisfies the triangle inequality. The maximum principle for harmonic functions together with Remark~\ref{rmk:jprops}~(iii) furthermore implies that $r(x,y) \geq 0$ with equality if and only if $x=y$. We obtain the well-known fact (see, e.g., \cite{kr}) that the effective resistance is a distance function on $\Gamma$.
\end{itemize}
\end{Remark}
\subsection{Computing $j$-functions} \label{sec:computeJ}

Let $G$ be an arbitrary model of $\Gamma$. One can explicitly compute the quantities $j_q(p,v) \in \RR$ for $q,p,v \in V(G)$ using linear algebra (see \cite[\S 3]{fm}) as follows: fix a labeling of $V(G)$ as before, and let $\Q$ be the corresponding Laplacian matrix. Let $\Q_q$ be the  $(n-1)\times(n-1)$ matrix obtained from $\Q$ by deleting the row and column corresponding to $q \in V(G)$ from $\Q$. The matrix $\Q_q$ is invertible, see e.g.\ Remark~\ref{rmk:weights}~(ii). Let $\L_q$ be the $n \times n$ matrix obtained from $\Q_q^{-1}$ by inserting zeros in the row and column corresponding to $q$. One can easily check that 
\begin{equation}\label{eq:Qq}
\Q\L_q = \I + \R_q  \, ,
\end{equation}
where $\I$ is the $n \times n$ identity matrix and $\R_q$ has all $-1$ entries in the row corresponding to $q$ and has zeros elsewhere. 
It follows from \eqref{eq:Qq} and the compatibility of various Laplacians in \S\ref{sec:laplacian} that:
\[
\L_q = (j_q(p,v))_{p,v \in V(G)} \, .
\]
\begin{Remark} \phantomsection \label{rmk:Lq}
\begin{itemize}
\item[]
\item[(i)] In Corollary~\ref{cor:jformula} we give another explicit formula for computing the $j$-function on weighted graphs.
\item[(ii)] Clearly $\R_q\Q = \mathbf{0}$, where $\mathbf{0}$ is the $n \times n$ zero matrix. Therefore $\L_q$ is a {\em generalized inverse} of $\Q$, in the sense that $\Q\L_q\Q = \Q$.
\item[(iii)] Computing $\L_q$ takes time at most $O(n^\omega)$, where
$\omega$ is the exponent for the matrix multiplication algorithm (currently $\omega < 2.38$).
\end{itemize}
\end{Remark}
\section{Energy pairings} \label{sec:energy}
In this section, we briefly study two useful pairings. We remark that both pairings can be defined and studied on larger vector spaces (see, e.g., \cite{br} for more general statements). Here we restrict our attention to those spaces that are relevant to our work, and give more explicit descriptions, statements, and proofs.
\begin{Definition}
Let $\Gamma$ be a metric graph. 
The {\em energy pairing} 
\[
\langle \cdot , \cdot \rangle_{\en} \colon \dmeasz(\Gamma) \times \dmeasz(\Gamma) \rightarrow \RR
\]
is defined by
\[
\langle \nu_1 , \nu_2 \rangle_{\en} \coloneqq \int_{\Gamma \times \Gamma} j_q (x,y) \mathrm{d}\nu_1(x) \mathrm{d}\nu_2(y) \, .
\]
for a fixed $q \in \Gamma$. 
\end{Definition}
If we want to clarify $\Gamma$, we will use the notation $\langle \nu_1 , \nu_2 \rangle_{\en}^{\Gamma}$ instead.

It follows from Lemma~\ref{lem:pairings} below that the energy pairing is indeed independent of the choice of $q$. A closely related concept is the {\em Dirichlet pairing}
\[
\langle \cdot , \cdot \rangle_{\Dir} \colon \PL(\Gamma) \times \PL(\Gamma) \rightarrow \RR
\] 
defined by
\[
\langle \phi_1 , \phi_2 \rangle_{\Dir} \coloneqq \int_{\Gamma} \phi_1 \Delta(\phi_2) = \int_{\Gamma} \phi_2 \Delta(\phi_1) \, .
\]
If $\nu_1=\Delta(\phi_1)$ and $\nu_2 = \Delta(\phi_2)$ then 
\[
\langle \nu_1 , \nu_2 \rangle_{\en} = \langle \phi_1 , \phi_2 \rangle_{\Dir} \, .
\]
Note that this equality does not depend on the choice of $\phi_1$ and $\phi_2$, which are well-defined only up to constant functions.

The energy pairing (and the Dirichlet pairing) can be computed using linear algebra: let $G$ be a model of $\Gamma$ such that $\nu_1, \nu_2 \in \dmeasz(G)$. Then $\langle \nu_1 , \nu_2 \rangle_{\en}$ can be computed using the (combinatorial) energy pairing
\[
\langle \cdot , \cdot \rangle_{\en} \colon \dmeasz(G) \times \dmeasz(G) \rightarrow \RR
\]
defined by
\begin{equation} \label{eq:en_pair}
\langle \nu_1 , \nu_2 \rangle_{\en} \coloneqq \sum_{p,v \in V(G)} \nu_1(p) j_q(p,v) \nu_2(v) = [\nu_1]^{\T} \L_q [\nu_2] \, .
\end{equation}
Likewise, let $G$ be a model of $\Gamma$ such that $V(G)$ contains all those points of $\Gamma$ at which $\phi_1$ or $\phi_2$ changes slopes. Let $\psi_i \in \Mc(G)$ denote the function obtained from $\phi_i$ by restriction. Then $\langle \phi_1 , \phi_2 \rangle_{\Dir}$ can be computed using the (combinatorial) Dirichlet pairing
\[
\langle \cdot , \cdot \rangle_{\Dir} \colon \Mc(G) \times \Mc(G) \rightarrow \RR
\]
defined by 
\[
\langle \psi_1 , \psi_2 \rangle_{\Dir} \coloneqq [\psi_1]^{\T} \Q [\psi_2] \, . 
\]

\begin{Lemma} \label{lem:pairings}
Let $G$ be a weighted graph. Fix a labeling of $V(G)$ and let $\L$ be any generalized inverse of the Laplacian matrix $\Q$ (i.e. $\Q\L\Q = \Q$). Then the symmetric bilinear form on $\dmeasz(G)$ defined by
\[
\langle \nu_1 , \nu_2 \rangle_{\en} = [\nu_1]^{\T} \L [\nu_2]
\]
is independent of the choice of the generalized inverse $\L$, and is positive definite.
\end{Lemma}

\begin{proof}
Independence from $\L$ follows from the fact that, if $[\nu_i] = \Q[\psi_i]$ then 
\[
\langle \nu_1 , \nu_2 \rangle_{\en} = [\nu_1]^{\T} \L [\nu_2] = [\psi_1]^{\T} \Q \L \Q [\psi_2] = [\psi_1]^{\T} \Q [\psi_2] = \langle \psi_1 , \psi_2 \rangle_{\Dir} \, .
\]
Positive definiteness follows from the factorization \eqref{eq:factor} (see also \cite[Lemma 3.5]{fm}): let $[\nu] = \Q[\psi]$. Then
\[
\langle \nu , \nu \rangle_{\en} = [\psi]^{\T} \Q [\psi] = \|\D^{-\frac{1}{2}}\B^{\T} [\psi] \|_2 \, .
\]
The kernel of $\D^{-\frac{1}{2}}\B^{\T}$ is the space of constant functions in $\Mc(G)$.
\end{proof}

\begin{Remark}
A canonical choice for the generalized inverse of $\Q$ in Lemma~\ref{lem:pairings} is the `Moore--Penrose pseudoinverse' $\Q^+$. In fact, it is straightforward to check $\Q^+ = \frac{1}{n}\sum_{q \in V(G)} \L_q$ (see \cite[Construction 3.2, Construction 3.3]{fm}). In what follows, we find it more natural to work with $\L_q$'s directly.
\end{Remark}

The energy and Dirichlet pairings have the following interpretation in the language of electrical networks: consider $\nu \in \dmeasz(G)$ as an external `current source' attached to the network $G$. Then $\langle \nu , \nu \rangle_{\en}$ is precisely the total energy dissipated (per unit time) in $G$.

\section{Cross ratios} \label{sec:cross}
In this section, we introduce the notion of cross ratios for metric graphs. We remark that this notion is already mentioned (in passing) in \cite[Remark B.12]{BR_Book}. We establish some basic properties of these cross ratios and provide some basic examples.

\begin{Definition} \label{def:xi}
Let $\Gamma$ be a metric graph and fix $q \in \Gamma$. We define the {\em cross ratio} function (with respect to the base point $q$) $\xi_q \colon \Gamma^4 \rightarrow \RR$  by 
\[
\xi_q(x,y , z,w) \coloneqq j_q(x,z) +j_q(y,w) - j_q(x,w) - j_q(y,z) \,.
\]
If we want to clarify $\Gamma$, we will use the notation $\xi_q(x,y , z,w ; \Gamma)$ instead.
\end{Definition}

\begin{Lemma} \phantomsection \label{lem:CrossProp}
\begin{itemize}
\item[]
\item[(a)] $\xi(x,y , z,w)  \coloneqq \xi_q(x,y , z,w)$ is independent of the choice of $q$.
\item[(b)] $\xi(x,y , z,w) = \xi( z,w,x,y)$.
\item[(c)] $\xi(y,x , z,w) = - \xi( x, y, z,w)$.
\end{itemize}
\end{Lemma}

\begin{proof}
Parts (b) and (c) are immediate from Definition~\ref{def:xi}.

Let $G$ be a model for $\Gamma$ such that $x, y, z, w, q_1, q_2 \in V(G)$. 
Note that (see \eqref{eq:en_pair}):
\[
\xi_{q_i}(x,y , z,w) = \langle \delta_x-\delta_y, \delta_z-\delta_w \rangle_{\en} \, .
\]
Part (a) then follows from the fact that the energy pairing is independent of the choice of the base point $q_i$, cf.\ Lemma~\ref{lem:pairings}. See also \cite[Remark B.12]{BR_Book} for an outline of a different proof of part (a). 
\end{proof}
\begin{Remark} \phantomsection \label{rmk:cross_en}
\begin{itemize}
\item[]
\item[(i)] Yet another proof of Lemma~\ref{lem:CrossProp}~(a) can be obtained from Lemma~\ref{lem:GromovProd} and an explicit computation. Namely, one finds the relation:
\begin{equation}\label{eq:cd}
-2\, \xi_{q}(x,y , z,w) = r(x,z)+r(y,w) - r(x,w) - r(y,z) \, .
\end{equation}
\item[(ii)] As is evident from the proof of Lemma~\ref{lem:CrossProp}, we could define the cross ratio with the more canonical expression:
\[
\xi(x,y , z,w) = \langle \delta_x-\delta_y, \delta_z-\delta_w \rangle_{\en} \, .
\]
See also Corollary~\ref{cor:crossformula} and Remark~\ref{rmk:metric_cross} for other explicit formulas for computing cross ratios on weighted graphs and on metric graphs.
\end{itemize}
\end{Remark}

\begin{Example}[Proof of Lemma~\ref{lem:GromovProd} using cross ratios] \label{ex:GromovProd}
Let us compute the cross ratio $\xi(x,y,x,y)$ with respect to two different base points: 
\[
\begin{aligned}
\xi_x(x,y,x,y) &= j_x(x,x)+j_x(y,y) - j_x(x,y) - j_x(y,x) \\
&= j_x(y,y) = r(x,y) \, .
\end{aligned}
\] 
\[
\begin{aligned}
\xi_z(x,y,x,y) &= j_z(x,x)+j_z(y,y) - j_z(x,y) - j_z(y,x) \\
&= r(x,z)+r(y,z) - 2j_z(x,y)\, .
\end{aligned}
\] 
By Lemma~\ref{lem:CrossProp}~(a), we must have $\xi_x(x,y,x,y)=\xi_z(x,y,x,y)$ and therefore
\[j_z(x,y) = \frac{1}{2} \left(r(x,z)+r(y,z) - r(x,y)\right) \, .\]
\end{Example}

\begin{Example}[Reciprocity theorem in electrical networks] \label{ex:recip}
By Lemma~\ref{lem:CrossProp}~(a), we have $\xi_y(x,y, z, w) = \xi_w(x,y, z, w)$. Therefore
\[j_y(x,z) - j_y(x,w) = j_w(x,z) - j_w(y,z) \, . \]
This is the celebrated `reciprocity theorem' for electrical networks (see, e.g., \cite[\S 5.3]{SB1959}, \cite[\S 17.2]{circuit2011}, \cite[Theorem 4]{Tetali}, \cite[Theorem 8]{bffourier}): 
informally, the location of the current source and the resulting voltage may be interchanged without a change in voltage. See Figure~\ref{fig:reciprocity}.
\ctikzset{bipoles/length=.5cm}
\begin{figure}[h!]
\begin{tikzpicture}
\node [cloud, draw,cloud puffs=10.2,cloud puff arc=120, aspect=2, inner ysep=1em, gray] (cloud) at (0, 0) {};
	\fill[black] (-.75,.3) circle (.1);
	\fill[black] (-.75,-.3) circle (.1);
	\fill[black] (.75,.3) circle (.1);
	\fill[black] (.75,-.3) circle (.1);
	\node at (.5, .3) {\tiny{$x$}};
	\node at (.5, -.3) {\tiny{$y$}};
	\node at (-.5, .3) {\tiny{$z$}};
	\node at (-.5, -.3) {\tiny{$w$}};
	\draw (.75,-.3) to (1.8,-.3) to[dcisource] (1.8, .3) to (.75,.3);
	\node at (2.1, 0) {\tiny{$I$}};
	\draw (-1.8,.3) [o-] to (-.75,.3);
	\draw  (-1.8,-.3) [o-] to (-.75,-.3);
	\node at (-2, .3) {\tiny{$+$}};
	\node at (-2, -.3) {\tiny{$-$}};
	\node at (-2, 0) {\tiny{$V$}};
\end{tikzpicture}
$\quad \quad$
\begin{tikzpicture}
\node [cloud, draw,cloud puffs=10.2,cloud puff arc=120, aspect=2, inner ysep=1em, gray] (cloud) at (0, 0) {};
	\fill[black] (-.75,.3) circle (.1);
	\fill[black] (-.75,-.3) circle (.1);
	\fill[black] (.75,.3) circle (.1);
	\fill[black] (.75,-.3) circle (.1);
	\node at (.5, .3) {\tiny{$x$}};
	\node at (.5, -.3) {\tiny{$y$}};
	\node at (-.5, .3) {\tiny{$z$}};
	\node at (-.5, -.3) {\tiny{$w$}};
	\draw (-.75,-.3) to (-1.8,-.3) to[dcisource] (-1.8, .3) to (-.75,.3);
	\node at (-2.1, 0) {\tiny{$I$}};
	\draw (.75,.3) to (1.8,.3) [-o];
	\draw  (.75,-.3) to (1.8,-.3) [-o];
	\node at (2, .3) {\tiny{$+$}};
	\node at (2, -.3) {\tiny{$-$}};
	\node at (2, 0) {\tiny{$V$}};
\end{tikzpicture}
\caption{
Reciprocity theorem for electrical networks: 
(left) $V= \left(j_y(x,z) - j_y(x,w)\right) I$, 
(right) $V =\left(j_w(x,z) - j_w(y,z)\right) I$.
}\label{fig:reciprocity}
\end{figure}
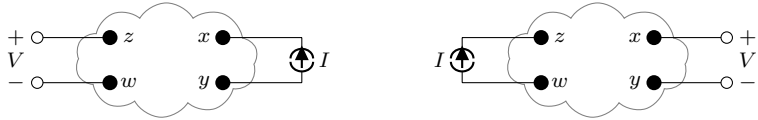
\end{Example}
\begin{Remark}\label{rmk:gen_recip}
One may think of Lemma~\ref{lem:pairings} as a vast generalization of the reciprocity theorem. 
\end{Remark}

\section{Projections} \label{sec:projs}

Throughout this section, we fix a model $G$ for a metric graph $\Gamma$, and fix an orientation $\Oc$ on $G$.
It is convenient to define the real $1$-chains by
\[C_1(G,\RR) \coloneqq \frac{\bigoplus_{e \in \mathbb{E}(G)} \RR e }{ \langle e+\bar{e} \colon e \in \Oc \rangle} \simeq \bigoplus_{e \in \Oc}  \RR e \,.\] 
So, for $e \in \mathbb{E}(G)$, the above presentation implies $\bar{e} = -e$ inside $C_1(G,\RR)$. Note that $\Oc$ is a basis for $C_1(G,\RR)$. 

\begin{Definition}\label{def:assChain}
For any subset $\MA \subseteq \EE(G)$, we define its {\em associated $1$-chain} as 
\[
\boldsymbol{\gamma}_\MA = \sum_{e \in \MA} e =  \sum_{e\in \Oc} \sign(\MA, e) \, e \, .
\]
where 
\[
\sign(\MA, e) =
\begin{cases}
+1 &\text{ if } e \in \MA \\
-1 &\text{ if } \bar{e} \in \MA \\
0  &\text{ otherwise. }
\end{cases}
\]
\end{Definition}
We will work with the usual definition of real $0$-chains:
\[C_0(G,\RR)\coloneqq \bigoplus_{v \in V(G)}  \RR v \, .\]

Let $\Oc=\{e_1,\ldots,e_m\}$ be a labeling of the orientation $\Oc$ of $G$.
The real vector space $C_1(G,\RR)$ has a canonical inner product 
\[
[ \cdot,\cdot  ] \colon C_1(G,\RR) \times C_1(G,\RR) \rightarrow \RR
\]
defined by $[ e_i,e_j ]=\delta_i(j)\ell(e_i)$.

Consider the usual boundary map
$\partial \colon C_1(G, \RR) \rightarrow C_0(G, \RR)$ defined by
$\partial(e)= e^+ - e^-$.
The first homology group coincides with the space of $1$-cycles 
\[H_1(G, \RR) = \Ker \partial \, . \] 
The inner product $[ \cdot,\cdot  ]$ restricts to an inner product, also denoted by $[ \cdot,\cdot  ]$, on $H_1(G,\RR)$. 
It is easy to check that the pair $\left(H_1(G,\RR),[ \cdot,\cdot  ]\right)$ is a canonical inner product space associated to $\Gamma$; it is independent of the choice of the model $G$.

\begin{Remark} \phantomsection \label{rmk:mat}
\begin{itemize}
\item[]
\item[(i)] The labeling $\Oc=\{e_1,\ldots,e_m\}$ fixes an isomorphism 
\[
[\,\cdot\,] \colon C_1(G,\RR) \xrightarrow{\sim} \RR^m \, .
\]
\item[(ii)] The incidence matrix $\B$ in \S\ref{sec:compatible} is precisely the matrix of $\partial$ with respect to bases $\Oc$ for $C_1(G, \RR)$, and $V(G)$ for $C_0(G, \RR)$. 
\item[(iii)] The matrix $\D$ in \S\ref{sec:compatible} is precisely the Gram matrix associated to the pair $\left(C_1(G, \RR), [\cdot, \cdot]\right)$ with respect to the basis $\Oc$.
\end{itemize}
\end{Remark}

We are interested in the two orthogonal projection maps 
\[
\pi \colon C_1(G,\RR) \twoheadrightarrow H_1(G, \RR) \, ,\] 
\[
\pi' \colon C_1(G,\RR) \twoheadrightarrow H_1(G, \RR)^{\perp}
\] 
and their matrix representations. Here $H_1(G, \RR)^{\perp}$ denotes the orthogonal complement of $H_1(G, \RR) \subseteq C_1(G, \RR)$ with respect to $[ \cdot , \cdot ]$.

\subsection{Electrical network problems} \label{sec:networkproblems}
The original motivation for computing projection matrices comes from electrical network theory. 

\subsubsection{The Kirchhoff problem} 
Consider the {\em Kirchhoff problem}: 
\begin{center}
Given $\cc \in C_1(G, \RR)$, find $\ii \in H_1(G, \RR)^{\perp}$ such that $\cc - \ii \in H_1(G, \RR)$.
\end{center}
Here $\cc$ should be thought of as an {\em external current source}, and $\ii$ should be thought of as the induced {\em internal current}. 

The condition $\cc - \ii \in H_1(G, \RR)$ is precisely the {\em Kirchhoff's current law}. The condition $\ii \in H_1(G, \RR)^{\perp}$ is precisely the {\em Kirchhoff's voltage law}. These laws are equivalent to computing the orthogonal decomposition $\cc = \ii + (\cc-\ii)$, so the solution is provided by computing 
\[\ii = \pi'(\cc)\, .\]

\begin{Remark}
The contribution of the external current source $\cc$ only depends on its boundary $\partial \cc$. It is customary to only refer to $\partial \cc$ as the external current source. See, e.g., Remark~\ref{rmk:jelectric}.
\end{Remark}

\subsubsection{Coboundaries and Ohm's law} \label{sec:Ohm}
The space of coboundaries is, by definition, 
\[{\rm Im}\left( d \colon C^0(G, \RR) \rightarrow C^1(G, \RR)\right) \, .\]
We may define an isomorphism from $C_1(G, \RR)$ to $C^1(G, \RR)$ using the bilinear form $[\cdot , \cdot]$. More precisely, we may think of $C^1(G, \RR) \simeq \bigoplus \RR e^*$ with $e^*(e)= \ell(e)$. The isomorphism is defined by  ${e}/{\ell(e)} \mapsto e^*$. Under this isomorphism $H_1(G, \RR)^\perp$ is identified with the space of coboundaries. 

Under this identification $\ii$ corresponds to a coboundary element $\vv$, referred to as {\em internal voltage}. Explicitly, the internal voltage $\vv = \sum_{e \in \Oc} v(e)\, e^*$ is identified with $\sum_{e \in \Oc} v(e)/\ell(e)\, e = \ii = \sum_{e \in \Oc} i(e)\, e$ so $v(e) = \ell(e) i(e)$. This is {\em Ohm's law}.

\subsubsection{The Dirichlet problem} \label{sec:Dirichlet}
Since $\vv$ in \S\ref{sec:Ohm} is a coboundary, $\vv = d \psi$ for some $\psi \in C^0(G, \RR)$ (well-defined up to a constant function). The $0$-cochain $\psi$ is called the {\em potential} associated to $\vv$. One might be interested to directly compute this $\psi$ given the external source $\partial \cc$. It is easy to check this problem boils down to solving the {\em Dirichlet problem}: 
\[\Delta(\psi) = \partial \cc \, .\] 
Here $\Delta$ is as in \S\ref{sec:comb_lap} and $\partial \cc$ is thought of as a discrete measure on $V(G)$, after identifying $v \in C_0(G, \RR)$ with $\delta_v \in \dmeas(G)$. The $j$-function defined in \S\ref{sec:jfunc} provides the fundamental solutions for this Dirichlet problem: if $\partial \cc = \sum_{v \in V(G)} a_v \delta_v$ then $\psi = \sum_{v \in V(G)} a_v  j_{q}(\cdot\, ,  v)$ for any fixed $q \in V(G)$.

\subsection{Projections using spanning trees} \label{sec:projST}
Before presenting the projection formulas in terms of cross ratios, we will review Kirchhoff's beautiful description of these projections (in the basis $\Oc$) as a certain average over spanning trees. This was introduced in the seminal paper \cite{Kirchhoff}. 

Recall that a spanning tree $T$ of $G$ is a maximal subset of $E(G)$ that contains no circuit (closed simple path). Equivalently, $T$ is a minimal subset of $E(G)$ that connects all vertices.

\subsubsection{Fundamental circuits and cocircuits}
Let $T$ be a spanning tree of $G$.
Each (unoriented) edge $e \not\in T$ determines a {\em fundamental circuit}, i.e. a unique circuit $\C(T, e) \subseteq E(G)$ in $T \cup e$. Let $e$ also denote the oriented edge in the fixed orientation $\Oc$ corresponding to the (unoriented) edge $e$. Note that every edge in $\C(T, e)$ comes with a preferred choice of orientation, namely the orientations that agree with the direction of the oriented edge $e\in \Oc$ as one travels along the circuit. 

\begin{Definition} \phantomsection \label{def:fundcircuits}
\begin{itemize}
\item[]
\item[(i)] For $e \not \in T$, we let $\circu(T, e)$ be the associated $1$-chain of $\C(T, e)$ endowed with its preferred orientation (Definition~\ref{def:assChain}).
For $e \in T$ we define $\circu(T, e) = 0$.
\item[(ii)]We let $\M_T$ be the $m \times m$ matrix whose columns are $[\circu(T, e)]$ for $e \in \Oc$, where $[ \cdot]$ is as in Remark~\ref{rmk:mat}~(i).
\end{itemize}
\end{Definition}
It is well-known (and easy to check) that $\{\circu(T, e) \colon e \not \in T \}$ forms a basis for $H_1(G, \RR)$.

Each (unoriented) edge $e \in T$ determines a {\em fundamental cocircuit}, i.e. the unique minimal subset $\Bc(T, e) \subseteq (E(G) \backslash T) \cup e$ such that $E(G) \backslash \Bc(T, e)$ is disconnected.
Let $e$ also denote the oriented edge in the fixed orientation $\Oc$ corresponding to the (unoriented) edge $e$. Note that every edge in $\Bc(T, e)$ comes with a preferred choice of orientation, namely the orientation that agrees with the direction of $e\in \Oc$ in the cut-set $\Bc(T, e)$. 

\begin{Definition} \phantomsection \label{def:fundcocircuits}
\begin{itemize}
\item[]
\item[(i)] For $e\in T$, we let $\cocirc(T, e)$ be the associated $1$-chain of $\Bc(T, e)$ endowed with its preferred orientation (Definition~\ref{def:assChain}).
For $e \not \in T$ we define $\circu(T, e) = 0$.
\item[(ii)] We let $\N_T$ be the $m \times m$ matrix whose columns are $[\cocirc(T, e)]$ for $e \in \Oc$, where $[ \cdot]$ is as in Remark~\ref{rmk:mat}~(i).
\end{itemize}
\end{Definition}

\subsubsection{Weights and coweights}
The {\em weight} of a spanning tree $T$ of $G$ is the product $
w(T) \coloneqq \prod_{e \not \in T}{\ell(e)}$. 
The {\em coweight} of a spanning tree $T$ of $G$ is the product $
w'(T) \coloneqq  \prod_{e \in T}{\ell^{-1}(e)}$.
The {\em weight} and {\em coweight} of $G$ are
$w(G) \coloneqq \sum_{T}{w(T)}$ and $w'(G) \coloneqq \sum_{T}{w'(T)}$,
where the sums are over all spanning trees of $G$.

\begin{Remark} \phantomsection \label{rmk:weights}
\begin{itemize}
\item[]
\item[(i)] For any spanning tree $T$ of $G$ we have ${w(T)}/{w(G)} = {w'(T)}/{w'(G)}$. Moreover, the quantity $w(G)$ depends only on the underlying metric graph $\Gamma$. This is not true of $w'(G)$. 
\item[(ii)] By  (Tutte's version of) the Kirchhoff's matrix tree theorem (\cite[Theorem VI.27]{tutte}), both $w(G)$ and $w'(G)$ can be expressed in terms of certain determinants. For example,  $w'(G) = \det(\Q_q)$, where $\Q_q$ is as defined in \S\ref{sec:computeJ}. These are simple consequences of the `Cauchy--Binet formula' for determinants. See \cite[Section 5]{abks} for more details and for a geometric (or tropical) proof.
\end{itemize}
\end{Remark}

\subsubsection{Kirchhoff's projection formulas}
Consider the following matrix averages:
\[
\P = \sum_{T}{\frac{w(T)}{w(G)} \M_T} \quad , \quad \P' = \sum_{T}{\frac{w'(T)}{w'(G)} \N_T} \, ,
\]
the sums being over all spanning trees $T$ of $G$. 
\begin{Proposition}[Kirchhoff] \phantomsection \label{prop:kirchhoffproj}
\begin{itemize}
\item[]
\item[(a)] The matrix of $\pi \colon C_1(G,\RR) \twoheadrightarrow H_1(G, \RR)$, with respect to $\Oc$, is $\P$.
\item[(b)] The matrix of $\pi' \colon C_1(G,\RR) \twoheadrightarrow H_1(G, \RR)^{\perp}$, with respect to $\Oc$, is $(\P')^{\T}$. 
\end{itemize}
\end{Proposition}
\begin{proof}
By \cite[Proposition~15.2]{Biggs} we know the matrix of $\pi'$ is $\D^{-1}\P' \D$. Since $\D^{-1}\P'$ is symmetric (\cite[Proposition~15.1]{Biggs}) we have 
\[\D^{-1}\P' \D = (\P')^{\T} \D^{-1} \D = (\P')^{\T}\, .\] 
This proves part (b). 
Part (a) follows from part (b) and the fact that $\I - (\P')^{\T} = \P$. See the computation in the first paragraph of \cite[\S16]{Biggs}.
\end{proof}
We note that a given graph $G$ can have many (super exponential number of) spanning trees, so computations using Proposition~\ref{prop:kirchhoffproj} are highly inefficient.

\subsection{Projections using cross ratios} \label{sec:projCR}
We now show that our projection matrices have expressions in terms of cross ratios. They are efficient for computations, and useful for proving theorems.

Let $\Xib$ be the $m \times m$ {\em matrix of cross ratios}:
\[
\Xib \coloneqq \left(\xi(e^-, e^+ , f^- , f^+) \right)_{e,f \in \Oc} \, . 
\]
It follows from Lemma~\ref{lem:pairings} and Remark~\ref{rmk:cross_en}~(ii) that 
\begin{equation} \label{eq:XiMatrix}
\Xib = \B^{\T} \L \B 
\end{equation}
for any generalized inverse $\L$ of the Laplacian matrix $\Q$.

\begin{Proposition}\label{prop:crossproj}
Let $\D$ be as in \S\ref{sec:compatible}.
\begin{itemize}
\item[(a)] The matrix of $\pi \colon C_1(G,\RR) \twoheadrightarrow H_1(G, \RR)$, with respect to $\Oc$, is 
$\I- \D^{-1} \Xib$.
\item[(b)] The matrix of $\pi' \colon C_1(G,\RR) \twoheadrightarrow H_1(G, \RR)^{\perp}$, with respect to $\Oc$, is 
$\D^{-1} \Xib$.
\end{itemize}
\end{Proposition}
\begin{Remark} \label{rmk:efficient}
One can compute $\Xib$ and both these projection matrices in time at most $O(n^\omega)$, where $\omega$ is the exponent for the matrix multiplication algorithm (currently $\omega < 2.38$). See Remark~\ref{rmk:Lq}~(iii).
\end{Remark}
\begin{proof}
It suffices to prove (b), which follows from a straightforward linear algebraic argument. We identify $C_1(G, \RR)$ with $\RR^m$ using $\Oc$. Recall $H_1(G, \RR) = \Ker \partial = \Ker \B$. Therefore (see Remark~\ref{rmk:mat}) we have $H_1(G, \RR)^{\perp} = {\rm Im}\, \D^{-1}\B^{\T}$. 

For $\bf \in \RR^m$, let $\hat{\bf} = \D^{-1}\B^{\T} {\xb}$ denote its orthogonal projection onto ${\rm Im}\, \D^{-1}\B^{\T}$. From $\bf - \hat{\bf} \in \Ker \B$ we obtain the Dirichlet problem $\Q \xb = \B \bf$ which has $\xb = \L_q \B \bf$ as a solution (see \S\ref{sec:Dirichlet}). Therefore, by \eqref{eq:XiMatrix}, we have 
\[\hat{\bf} = \D^{-1}\B^{\T} \L_q \B \bf = \D^{-1} \Xib\, \bf \, .\]
\end{proof}

The following is a restatement of Proposition~\ref{prop:crossproj} in a more canonical language.
\begin{Corollary} \label{cor:Fs}
For any $f \in \Oc$ we have
\begin{itemize}
\item[(a)] $\pi(f) = \sum_{e \in \Oc} {\F(e,f) e}$, where 
\[
\F(e,f) \coloneqq
\begin{cases}
 1-{r(e^-, e^+)}/{\ell(e)}  &\text{ if } e=f\\
-{\xi(e^-, e^+ ,f^-, f^+)}/{\ell(e)}  &\text{ if } e \ne f \, .
\end{cases} 
\]
\item[(b)] $\pi'(f) = \sum_{e \in \Oc} {\F'(e,f) e}$, where 
\[
\F'(e,f) =
{\xi(e^-, e^+ ,f^-, f^+)}/{\ell(e)} \, .
\]
\end{itemize}
\end{Corollary}

\subsection{Relations and consequences}  \label{sec:2projs}
The two different descriptions of projection matrices (Proposition~\ref{prop:kirchhoffproj} and Proposition~\ref{prop:crossproj}) have some important consequences.
\begin{itemize} 
\item[(i)] We have equalities 
\begin{equation} \label{eq:Ps_Cross}
\P= \I- \D^{-1} \Xib \quad , \quad \P' = \Xib \D^{-1}\, .
\end{equation}
\item[(ii)] Using \eqref{eq:XiMatrix} and \eqref{eq:Ps_Cross} we obtain 
\[
\P' = \B^{\T} \L \B  \D^{-1}
\]
This refines (and generalizes) the `canonical factorization' of Biggs in \cite[\S8 and \S15]{Biggs}.
\item[(iii)] The {\em Foster coefficient} of $e \in \Oc$ is, by definition,
\[\F(e) \coloneqq \F(e,e) = 1-{r(e^-, e^+)}/{\ell(e)}\, .\] 
Clearly $\F(e) = \F(\bar{e})$, so the Foster coefficient is also defined for $ e \in E(G)$. It measures the probability ${\rm Pr}\{e \not\in {\rm T}\}$, where ${\rm T}$ is a weighted uniform spanning tree. 
It is easy to see that $\sum_{e \in E(G)}\F(e) = \dim_{\RR} H_1(G, \RR)$. In fact, both sides of this equality represent the trace of the orthogonal projection matrix $\P$. This is the theorem of Ronald Foster in \cite{Foster}. See also \cite{Flanders}, \cite[Theorem 6]{Tetali}, \cite[Corollary 6.5]{bf}.

\item[(iv)] We have 
\[\xi_{f^-}(e^-, e^+ ,f^-, f^+) = j_{f^-}(e^+, f^+) - j_{f^-}(e^-, f^+)\, .\]
Therefore, $\F'(e,f)$ can be interpreted as the current that flows across $e$ when a unit current is imposed between the endpoints of $f$. In this way, we recover the well-known description as a `transfer--current matrix' for $\pi'$ in probability theory (see, e.g., \cite{BP}, \cite[\S2.4, \S4.2]{LP}, \cite[\S4.3.2]{HKPYV}, \cite[\S4]{BLPS}). The `transfer--current theorem' states that the weighted uniform spanning tree of $G$ is a determinantal point process on $E(G)$ with kernel $\pi'$. 

\end{itemize}

\subsection{Energy pairing and cross ratios using projections}
We have already seen the entries of projection matrices are computed from certain cross ratios. We now show that an arbitrary cross ratio can be computed using these projection matrices.

A {\em path} $\gamma$ in $G$ is an alternating sequence of vertices $v_i$ and oriented edges $e_i$,
\[
v_0, e_0, v_1, e_1, v_2, \ldots , v_{k-1}, e_{k-1}, v_k
\]
such that $e_i^- = v_i$ and $e_i^+ = v_{i+1}$. A {\em closed path} is one that starts and ends at the same vertex. One can associate a $1$-chain $\boldsymbol{\gamma}$ to the path $\gamma$ by applying Definition~\ref{def:assChain} to the set of oriented edges $\{e_0, \ldots, e_{k-1}\}$. By construction, we have $\partial \boldsymbol{\gamma} = \delta_{v_k} - \delta_{v_0}$.

More generally, for any $\nu \in \dmeasz(G)$, it is easy to see there exists $\boldsymbol{\gamma} \in C_1(G, \RR)$, well-defined up to an element of $H_1(G, \RR)$, such that $\partial (\boldsymbol{\gamma}) = \nu$. To see this, let 
$\nu = \sum_{v \in V(G)} a_v \delta_v$. Then for any fixed $q \in V(G)$ we have 
$\nu = \sum_{v \in V(G)} a_v (\delta_v- \delta_q)$. Let $\boldsymbol{\gamma}_{qv}$ be an arbitrary path in $G$ from $q$ to $v$. Then $\boldsymbol{\gamma} = \sum_{v \in V(G)} a_v \boldsymbol{\gamma}_{qv}$ has the property that $\partial (\boldsymbol{\gamma}) = \nu$.

\begin{Proposition} \label{prop:energyformula}
Let $\nu_1, \nu_2 \in \dmeasz(\Gamma)$. Fix a model $G$ compatible with $\nu_1, \nu_2$ and let $\boldsymbol{\gamma}_1 , \boldsymbol{\gamma}_2 \in C_1(G, \RR)$ be such that $\partial (\boldsymbol{\gamma}_i) = \nu_i$. Then
\[
\langle \nu_1, \nu_2 \rangle_{\en} = [\boldsymbol{\gamma}_1 , \pi'(\boldsymbol{\gamma}_2)] \, .
\]
Here $\pi' \colon C_1(G,\RR) \twoheadrightarrow H_1(G, \RR)^{\perp}$ denotes the orthogonal projection as before. 
\end{Proposition}

\begin{proof}
We use the basis $\Oc$ and do the computations with the help of corresponding matrices:
\[
\begin{aligned}
[\boldsymbol{\gamma}_{1}, \pi'(\boldsymbol{\gamma}_{2})] &= [\boldsymbol{\gamma}_{1}]^{\T} \D [\pi'(\boldsymbol{\gamma}_{2})] &\text{(Remark~\ref{rmk:mat}~(iii))}\\
&= [\boldsymbol{\gamma}_{1}]^{\T} \D \D^{-1}\Xib[\boldsymbol{\gamma}_{2}] &\text{(Proposition~\ref{prop:crossproj})}\\
&=[\boldsymbol{\gamma}_{1}]^{\T} \B^{\T} \L \B [\boldsymbol{\gamma}_{2}] &\text{\eqref{eq:XiMatrix}}\\
&=[\partial \boldsymbol{\gamma}_{1}]^{\T} \L [\partial \boldsymbol{\gamma}_{2}]&\text{(Remark~\ref{rmk:mat}~(ii))}\\
&=[\nu_1]^{\T} \L [\nu_2] \\
&=\langle \nu_1 , \nu_2\rangle_{\en}&\text{(Lemma~\ref{lem:pairings})}\, .
\end{aligned}
\]
\end{proof}

\begin{Corollary} \label{cor:crossformula}
Fix arbitrary paths in $G$ from $y$ to $x$, and from $w$ to $z$. Let $\boldsymbol{\gamma}_{yx}$ and $\boldsymbol{\gamma}_{wz}$ denote their associated $1$-chains. Then
\[
\xi(x,y, z, w) = [\boldsymbol{\gamma}_{yx}, \pi'(\boldsymbol{\gamma}_{wz})]
\]
\end{Corollary}
\begin{proof}
This follows from Proposition~\ref{prop:energyformula}, applied to $\nu_1 = \delta_x - \delta_y$ and $\nu_2 = \delta_z - \delta_w$. See Remark~\ref{rmk:cross_en}~(ii).
\end{proof}

An explicit integral formula for $j$-functions is given in \cite[Proposition 3.17]{sw}. The following is a discrete version of that result.
\begin{Corollary} \label{cor:jformula}
Fix arbitrary paths from $z$ to $x$, and from $z$ to $y$. Let $\boldsymbol{\gamma}_{zx}$ and $\boldsymbol{\gamma}_{zy}$ denote their associated $1$-chains. Then
\[
j_z(x,y) = [\boldsymbol{\gamma}_{zx}, \pi'(\boldsymbol{\gamma}_{zy})] \, .
\]
\end{Corollary}

\begin{proof}
This follows from Corollary~\ref{cor:crossformula}, applied to  $\xi(x,z, y, z) = j_z(x,y)$.
\end{proof}
\begin{Corollary} \label{cor:rformula}
Fix two arbitrary paths from $y$ to $x$. Let $\boldsymbol{\gamma}_{yx}$ and $\boldsymbol{\gamma}'_{yx}$ denote the associated $1$-chains. Then
\[
r(x,y) = [\boldsymbol{\gamma}'_{yx}, \pi'(\boldsymbol{\gamma}_{yx})] \, .
\]
\end{Corollary}
\begin{proof}
This follows from Corollary~\ref{cor:crossformula}, applied to  $\xi(x,y, x, y) = r(x,y)$.
\end{proof}
\begin{Remark}\label{rmk:Thomson}
Since $\pi'^2 = \pi'$ and $\pi'$ is self-adjoint with respect to $[ \cdot , \cdot]$, one can write the expression in Corollary~\ref{cor:rformula} as $r(x,y) = [\pi'(\boldsymbol{\gamma}_{yx}), \pi'(\boldsymbol{\gamma}_{yx})]$. So $r(x,y)$ is the {\em norm squared} of the projected vector $\pi'(\boldsymbol{\gamma}_{yx})$. This is equivalent to {\em Thomson's principle} for electrical networks (see \cite[\S18]{Biggs}). Proposition~\ref{prop:energyformula} may be thought of as a generalized version of Thomson's principle.
\end{Remark}

\begin{Remark} \label{rmk:metric_cross}
Corollary~\ref{cor:crossformula} can easily be proved assuming Corollary~\ref{cor:jformula} (using Definition~\ref{def:xi}). Similarly, one can use the explicit integral formula for $j$-functions in \cite[Proposition 3.17]{sw} to write down an explicit integral formula for cross ratios on metric graphs. Namely, for all $x,y,z,w \in \Gamma$, we have 
	\begin{equation}\label{eq:jformula}
	\xi(x,y,z,w)=\int_{\gamma_{yx}}\left(\omega_{\boldsymbol{\gamma}_{wz}}-\pi(\omega_{\boldsymbol{\gamma}_{wz}})\right) \, .
	\end{equation}
Here $\pi \colon \Omega^1(\Gamma) \rightarrow \mathcal{H}(\Gamma)$ denotes the orthogonal projection from the space of piecewise constant $1$-forms to the subspace of harmonic $1$-forms on $\Gamma$. For $p,q \in \Gamma$, the $1$-form $\omega_{\boldsymbol{\gamma}_{pq}}$ is associated to a piecewise linear path from $p$ to $q$. We refer to \cite[\S3]{sw} for more details.

\end{Remark}

\section{Rayleigh's laws} \label{sec:Rayleigh}
\subsection{Contractions} \label{sec:contractions}
Let $\Gamma$ be a metric graph. Let $e \subseteq \Gamma$ be an edge segment with boundary points $\partial e = \{e^-, e^+\}$. We denote by $\Gamma / e$ the quotient metric graph whose equivalence
classes are $e$ and all one point subsets $\{x\}$ for $x \notin e$. Geometrically, one is contracting (collapsing) $e$ to a single point $p_e$.
From the point of view of electrical networks, it is best to think of setting $\ell(e)=0$, which can be interpreted as `short-circuiting' the segment $e$. 

Let $G$ be a model of $\Gamma$ so that $e \in E(G)$. Possibly upon making a refinement of the vertex set underlying $G$ we can take $V/\{e^-,e^+\}$  as a vertex set of $\Gamma/e$, yielding a model $G/e$ of $\Gamma/e$ with the property that $E(G/e)$ is canonically identified with $E(G) \backslash e$. We consider $C_1(G/e, \RR)$ as a subspace of $C_1(G, \RR)$ via the natural map 
\[\iota \colon C_1(G/e, \RR) \hookrightarrow C_1(G, \RR)\, .\]
Let $\pi'_G \colon C_1(G, \RR) \twoheadrightarrow H_1(G, \RR)^{\perp}$ and $\pi'_{G/e} \colon C_1(G/e, \RR) \twoheadrightarrow H_1(G/e, \RR)^{\perp}$ denote the orthogonal projections.

 Let $W$ be the orthogonal complement of $\pi'_G(e)$ inside $H_1(G, \RR)^{\perp}$:  
\[ W = \span \{\pi'_G(e)\}^{\perp}  \cap H_1(G, \RR)^{\perp} \, .\]
Let $\Proj \colon H_1(G, \RR)^{\perp} \twoheadrightarrow W$ denote the corresponding orthogonal projection map. 
\begin{Proposition} \label{prop:diagram}
The map $\iota$ restricts to an isomorphism $\bar{\iota} \colon H_1(G/e, \RR)^{\perp} \xrightarrow{\sim} W$. Moreover the following diagram commutes:
\begin{equation} 
\xymatrix @R=.3in{
  {C_1(G/e, \RR)} \ar@{^{(}->}[r]^{\iota} \ar@{->>}[dd]_{\pi'_{G/e}} & {C_1(G, \RR)} \ar@{->>}[d]^{ \pi'_G} \\
  & {H_1(G, \RR)^{\perp}} \ar@{->>}[d]^{ \Proj} \\
  {H_1(G/e, \RR)^{\perp}} \ar[r]_{\bar{\iota}}^{\sim} & {W} }
  \end{equation}
\end{Proposition}
\begin{proof} Let $\pi_* \colon C_1(G,\RR) \twoheadrightarrow C_1(G/e,\RR)$ denote the canonical projection, and observe that $\Ker \pi_* = \span\{ e\}$. Write $K=\Ker(\Proj \circ \pi'_G)=\span \{ e \} + H_1(G,\RR)$ and $L=\pi_*^{-1}H_1(G/e,\RR)$.  The map $\pi_*$ induces an isomorphism $H_1(G,\RR) \isom H_1(G/e,\RR)$ upon restriction. Since clearly $K \subseteq L$ and $\dim K = m-n+2 = \dim L$ we find the equality $K=L$.  As $\iota$ splits $\pi_*$ we have $\iota^{-1}K = \iota^{-1}L=H_1(G/e,\RR)$. This shows that the inclusion $\iota \colon C_1(G/e,\RR) \hookrightarrow C_1(G,\RR)$ induces an injective map $\bar{\iota} \colon H_1(G/e,\RR)^\perp \hookrightarrow W$.   As $\dim H_1(G/e,\RR)^\perp = n-2 = \dim W$ we conclude that $\bar{\iota}$ is an isomorphism. This proves the proposition.
\end{proof}

\subsection{Generalized Rayleigh's laws}

We are now ready to state and prove our main results. 

\begin{Theorem}[Rayleigh's law for energy pairings]\label{thm:GenRayleighEnergy}
Let $\Gamma$ be a metric graph. Let $e$ be an edge segment of $\Gamma$, and let $\nu_1, \nu_2 \in \dmeasz(\Gamma)$.  Then
\[
\langle \nu_1, \nu_2 \rangle_{\en}^{\Gamma/e} = 
\langle \nu_1, \nu_2 \rangle_{\en}^{\Gamma} 
- \frac{\langle \nu_1, \delta_{e^+} - \delta_{e^-} \rangle_{\en}^{\Gamma} \ \,
\langle \delta_{e^+} - \delta_{e^-}, \nu_2 \rangle_{\en}^{\Gamma}}
{r(e^-, e^+; \Gamma)}\, .
\]
In particular, for $\nu\in \dmeasz(\Gamma)$, we have $\langle \nu, \nu \rangle_{\en}^{\Gamma/e}  \leq \langle \nu, \nu \rangle_{\en}^{\Gamma}$.
\end{Theorem}
\begin{proof}
Let $G$ be a model of $\Gamma$ determined by a vertex set $V$ so that $e \in E(G)$ and $V/\{e^-,e^+\}$ is a vertex set of $\Gamma/e$ (see \S\ref{sec:contractions}). Assume moreover that $V$ is  taken fine enough so that $\nu_1, \nu_2 \in \dmeasz(G)$. We choose an orientation $\Oc = \{e_1, \ldots , e_m\}$. Let $e$ also denote the corresponding oriented edge in $\Oc$. Let $\boldsymbol{\gamma}_1 , \boldsymbol{\gamma}_2 \in C_1(G, \RR)$ be such that $\partial (\boldsymbol{\gamma}_i) = \nu_i$. We have a well-defined model $G/e$ of $\Gamma/e$.

By Proposition~\ref{prop:energyformula}, we know:
\begin{equation}\label{eq:xicontract}
 \langle \nu_1, \nu_2 \rangle_{\en}^{\Gamma/e} = [\boldsymbol{\gamma}_{1}, \pi'_{G/e}(\boldsymbol{\gamma}_{2})]
\end{equation}
By Proposition \ref{prop:diagram}, we know $\pi'_{G/e}$ corresponds to $\Proj \circ \, \pi'_{G}$ via $\iota$. 

For $e_j \in \Oc$ we have:
\[
\begin{aligned}
\Proj \circ \pi'_G(e_j) &= \pi'_G(e_j) - \frac{[\pi'_G(e), \pi'_G(e_j)]}{[\pi'_G(e), \pi'_G(e)]}\pi'_G(e)\\
&= \pi'_G(e_j) - \frac{[e, \pi'_G(e_j)]}{[e, \pi'_G(e)]}\pi'_G(e)\\
&= \pi'_G(e_j) - \frac{\F'(e,e_j) \ell(e)}{\F'(e,e) \ell(e)}\pi'_G(e) \\
& = \pi'_G(e_j) - \frac{\F'(e,e_j) }{\F'(e,e) } \sum_{i=1}^m{\F'(e_i, e) e_i} \\
& = \pi'_G(e_j) - \frac{1}{r(e^-, e^+)} \sum_{i=1}^m{\frac{\xi(e_i^-, e_i^+ ,e^-, e^+) \xi(e^-, e^+ ,e_j^-, e_j^+)}{\ell(e_i)}\,e_i} \, .
\end{aligned}
\]
We used Corollary~\ref{cor:Fs} for the third, fourth, and fifth equalities. 

It follows from this computation that the matrix of $\pi'_{G/e} = \Proj \circ \, \pi'_G$, with respect to the basis $\Oc$, is given by
\begin{equation}\label{eq:S}
\Sb = \D^{-1}\Xib - \frac{1}{r(e^-, e^+)} \D^{-1} (\Xib [e]) (\Xib [e])^{\T}   \, .
\end{equation}
Recall from Remark~\ref{rmk:mat}~(i) that $[e]$ denotes the column vector with a $1$ on the row corresponding to $e$, and $0$'s everywhere else. 
The result now follows from \eqref{eq:xicontract} and the following straightforward matrix computation:
\[
\begin{aligned}
[\boldsymbol{\gamma}_{1}, \pi'_{G/e}(\boldsymbol{\gamma}_{2})]  
&=[\boldsymbol{\gamma}_{1}]^{\T} \D \left(\Sb[\boldsymbol{\gamma}_{2}]\right)\\
&=[\boldsymbol{\gamma}_{1}]^{\T}\Xib[\boldsymbol{\gamma}_{2}] - 
\frac{1}{r(e^-, e^+)}[\boldsymbol{\gamma}_{1}]^{\T}\left(\Xib [e]\right) \left(\Xib [e]\right)^{\T}[\boldsymbol{\gamma}_{2}]\\
&=[\boldsymbol{\gamma}_{1}]^{\T}\B^{\T} \L \B[\boldsymbol{\gamma}_{2}] - \frac{1}{r(e^-, e^+)}
\left([\boldsymbol{\gamma}_{1}]^{\T}\B^{\T} \L \B [e] \right) \left([e]^{\T} \B^{\T} \L \B [\boldsymbol{\gamma}_{2}]\right)\\
&=[\nu_1]^{\T} \L [\nu_2] 
 - \frac{1}{r(e^-, e^+)}\left([\nu_1]^{\T} \L[\delta_{e^+} - \delta_{e^-}]\right) \left( [\delta_{e^+} - \delta_{e^-}]^{\T} \L [\nu_2]\right)\\
&=\langle \nu_1, \nu_2 \rangle_{\en} 
- \frac{1}{r(e^-, e^+)}\left( \langle \nu_1, \delta_{e^+} - \delta_{e^-} \rangle_{\en} \ 
\langle \delta_{e^+} - \delta_{e^-}, \nu_2 \rangle_{\en}\right)\, .
\end{aligned}
\]
We used Remark~\ref{rmk:mat}~(iii), \eqref{eq:S}, \eqref{eq:XiMatrix}, Remark~\ref{rmk:mat}~(ii), Lemma~\ref{lem:pairings}, and Remark~\ref{rmk:cross_en}~(ii) in this computation.
\end{proof}

\begin{Corollary}[Rayleigh's law for cross ratios]\label{cor:GenRayleighCross}
Let $\Gamma$ be a metric graph. Let $e$ be an edge segment of $\Gamma$.  Then
\[
\xi (x,y,z,w ; \Gamma/e) = \xi(x,y , z,w ; \Gamma) -   \frac{\xi(x,y , e^-, e^+; \Gamma) \, \xi(z,w , e^-, e^+;\Gamma)}{r(e^-, e^+; \Gamma)} \, .
\]
\end{Corollary}
\begin{proof}
This is Theorem~\ref{thm:GenRayleighEnergy} applied to $\nu_1 = \delta_x - \delta_y$ and $\nu_1 = \delta_z - \delta_w$. See Lemma~\ref{lem:CrossProp} and Remark~\ref{rmk:cross_en}~(ii).
\end{proof}
\begin{Corollary}[Rayleigh's law for $j$-functions]\label{cor:GenRayleighJ} Let $\Gamma$ be a metric graph. Let $e$ be an edge segment of $\Gamma$. Then
\[
j_z(x,y ; \Gamma/e) = j_z(x,y ; \Gamma) -  \frac{\xi(x,z , e^-, e^+ ; \Gamma) \, \xi(y,z , e^-, e^+; \Gamma)}{r(e^-, e^+; \Gamma)} \, .
\]
\end{Corollary}
\begin{proof}
This is Corollary~\ref{cor:GenRayleighCross} applied to the $4$-tuple $(x,z,y,z)$.
\end{proof}
\begin{Corollary}[A quantitative Rayleigh's monotonicity law for resistances]\label{cor:EffRayleigh}
Let $\Gamma$ be a metric graph. Let $e$ be an edge segment of $\Gamma$. Then
\[
r(x,y ; \Gamma/e) =r(x,y ; \Gamma) -  \frac{\xi(x,y , e^-, e^+; \Gamma) ^2}{r(e^-, e^+; \Gamma)} \, .
\]
In particular, $ r(x,y ; \Gamma/e) \leq r(x,y ; \Gamma)$.
\end{Corollary}
\begin{proof}
This is Corollary~\ref{cor:GenRayleighCross} applied to the $4$-tuple $(x,y,x,y)$.
\end{proof}
The last statement in Corollary~\ref{cor:EffRayleigh} is equivalent to the classical  Rayleigh's monotonicity law.

\input{cross.bbl}


\end{document}